\documentclass[a4paper,12pt]{amsart}
\usepackage{amsfonts,amssymb,amscd}

\usepackage{xcolor}
\usepackage{epsf}
\usepackage{epsfig}
\usepackage{graphicx}
\usepackage{tikz-cd}

\newtheorem{thm}{Theorem}[section]
\newtheorem{lm}[thm]{Lemma}

\newtheorem{df}[thm]{Definition}
\newtheorem{prp}[thm]{Proposition}
\newtheorem{rmk}[thm]{Remark}

\newcommand{\NN}{\mathbb{N}}
\newcommand{\ZZ}{\mathbb{Z}}

\newcommand{\RR}{\mathbb{R}}
\newcommand{\CC}{\mathbb{C}}
\newcommand{\DD}{\mathbb{D}}

\newcommand{\PP}{\mathbb{P}}

\newcommand{\mc}[1]{\mathcal{#1}}

\newcommand{\rb}{\right)}
\newcommand{\lb}{\left(}

\newcommand{\rcb}{\right\rbrace }
\newcommand{\lcb}{\left\lbrace }
\newcommand{\rsb}{\right]}
\newcommand{\lsb}{\left[}
\newcommand{\rv}{\right|}
\newcommand{\lv}{\left|}

\newcommand{\Sch}{{\sl S}}

\newcommand{\fol}{\mathcal{F}}

\newcommand{\Ps}{{\sl P}}
\newcommand{\Dv}{{\sl D}}

\newcommand{\SL}{{\rm SL(2,\mathbb{C})}}
\newcommand{\PSLR}{{\rm PSL(2,\mathbb{R})}}
\newcommand{\PSLC}{{\rm PSL(2,\mathbb{C})}}
\newcommand{\AffC}{{\rm Aff(\mathbb{C})}}

\newcommand{\slC}{\mathfrak{sl}_{2}\mathbb{C}}

\begin{document}

\title[Uniformizable projective structures]{Uniformizable singular projective structures on Riemann surface orbifolds}

\author{Ahmed Elshafei, \, \, \, Julio C. Rebelo \, \, \, \& \, \, \, Helena Reis}
\address{}
\thanks{}

\subjclass[2010]{Primary 37F75, 57M50; Secondary 30F40}
\keywords{Uniformizable projective structures, singular points, monodromy}

\maketitle

\begin{abstract}
This paper is devoted to characterizing complex projective structures defined on Riemann surface orbifolds and giving rise to
injective developing maps defined on the monodromy covering of the surface (orbifold) in question.
The relevance of these structures
stems from several problems involving vector fields with uniform solutions as well as from problems about
``simultaneous uniformization'' for leaves of foliations by Riemann surfaces. In this paper, we first describe the
local structure of the mentioned projective structures showing, in particular, that they are locally bounded.
In the case of Riemann surface orbifolds of finite type, the previous result will then allow us to provide a detailed
global picture of these projective structures by exploiting their connection with the class of ``bounded
covering projective structures''.
\end{abstract}

\section{Introduction}

It is fair to say that the notion of {\it (singular) uniformizable projective structures} on Riemann surfaces
was first put forward
in \cite{guillotIHES} where the author conducted a very complete study of classical Halphen systems by exploiting
their connection with the Lie algebra $\slC$. Whereas the issue was already implicit in the classical literature,
in particular in the work of Schwarz about triangle functions and solutions of Gauss hypergeometric equation,
it is only in \cite{guillotIHES} that the
relevance of the underlying projective structure is made explicit and effectively used to
clarify the rather non-trivial dynamics
of vector fields with single-valued solutions (aka semicomplete vector fields, see \cite{adolfoandI})
as those in Halphen systems. Similar issues
also arise for more general automorphic functions as it will be seen.

Basically, a singular projective structure $\Ps$ on a Riemann surface orbifold is said to be {\it uniformizable} if the
developing map associated with the {\it monodromy covering}\, is one-to-one (see Section~\ref{sec:Basic_Schwarz} for
accurate definitions).
The interest of these projective structures is further emphasized in some recent papers. For example,
in \cite{adolfoandI}, a crucial part of the argument relies on the existence of uniformizable affine structures on
the polar divisor of the vector fields in question. Also, the systematic use of uniformizable projective structures
allows for important generalizations of Halphen vector fields, see \cite{adolfoHalphenfractal}, \cite{A_thesis}
and \cite{paper-A}. Additional (potential) applications of these structures to the general problem of ``simultaneous
uniformization'' for foliations will be indicated later. In any event, to have a feeling for the interest of this notion,
we may simply consider an action of $\SL$ and a three-dimensional orbit $\mathcal{L}$ of it.
Consider now the standard one-parameter subgroups
$g, h_-$, and $h_+$ of
of $\SL$ represented respectively by the ``geodesic flow'', the ``negative horocyclic flow'',
and the ``positive horocyclic flow''. The semi-direct product of $g$ and $h_-$ yields an action of the affine
group $\AffC$ on $\mathcal{L}$ whose orbits are transverse to $h_+$ in $\mathcal{L}$. In particular, $\mathcal{L}$ is equipped with a foliation
$\fol$ whose leaves are the orbits of $\AffC$. Under these conditions, and wild as it may be, the leaf space of $\fol$ is
naturally endowed with a projective structure arising from the flow of $h_+$ which, in fact,
is uniformizable since the action of $\SL$ is globally defined. Here, it is worth mentioning that similar projective structures
have also been exploited in previous works by E. Ghys involving Anosov flows, see
for example \cite{ghys1} and \cite{ghys2}.

For the time being, it suffices to think that what precedes serves as motivation for us to try and describe
singular uniformizable projective structures on Riemann surfaces, or more generally, on Riemann surface orbifolds.
This problem is subtle already in the case of compact Riemann surfaces (orbifolds) $\mc{L}$ since Bers simultaneous
uniformization theorem ensures that the space of uniformizable projective structures contains a copy of
a suitable Teichmuller space. This is in stark contrast with the special case of (singular) uniformizable affine structures
as exploited in \cite{adolfoandI}: on compact Riemann surface orbifolds there are only a handful of the latter structures.

In terms of classifying (singular) uniformizable projective structures, the main results obtained in this
paper - ordered from local to global - are listed in the sequel. We also note that Theorem~A below can naturally
be viewed as an extension of the so-called ``Fundamental Lemma'' in \cite{adolfoandI} from the setting of
``vector fields/uniformizable affine structures'' to the context of ``uniformizable projective structures''.

\vspace{0.2cm}

\noindent {\bf Theorem A}. {\sl Let $\Ps$ be a uniformizable projective structure on a
	punctured neighborhhod of $0 \in \CC$ and denote by $\xi (z) dz^2$ the corresponding Schwarzian differential.
	Then one of the following holds:
	\begin{itemize}
		\item The local monodromy of $\Ps$ around $0 \in \CC$ is an elliptic element of finite order~$k \in \ZZ$.
		Moreover, the lift of $\Ps$ to the ramified $k$-sheet covering $\DD_\mu$ of $(\CC,0)$ extends holomorphically
		to all of $\DD_\mu$.
		
		\item The local monodromy of $\Ps$ around $0 \in \CC$ is a parabolic element. Then $\xi (z)$ has
		a simple pole at $0 \in \CC$. Moreover, if the monodromy is given by $z \mapsto z +c$, then the residue of $\xi$ at $0 \in \CC$
		(coefficient of the term $z^{-1}$) in $\xi$ is given by $-c/(a_{-1}\pi i)$ where $a_{-1}$
		is the coefficient of $z^{-1}$ in the Laurent extension of the holomorphic function $g$ on $\DD^*$ cf. below.
	\end{itemize}
}

\vspace{0.1cm}

As to the first item of Theorem~A, we note that it is easy to obtain a ``ramified'' local expression for
$\xi$ out of the fact that its lift to the coordinate $y$, $z=y^k$, is holomorphic at the origin. A
useful consequence of Theorem~A is Corollary~B below. To state this result,
recall that every projective structure
on a Riemann surface (orbifold) $\mc{L}$ gives rise to a holomorphic quadratic differential $\omega = \xi (z) dz^2$
on $\mc{L}$ called its {\it Schwarzian differential}, see Section~\ref{sec:Basic_Schwarz}.
In turn, if $\omega = \xi (z) dz^2$ is a quadratic differential
on $\mc{L}$, we define the $L^{\infty}$-(hyperbolic) norm of $\xi (z) dz^2$ as follows. Letting $\widetilde{\omega}
= \widetilde{\xi} d\widetilde{z}^2$ denote the lift of $\omega$ to the unit disc $\DD$ viewed as the universal
covering of $\mc{L}$, we pose
$$
\Vert \omega \Vert_{\infty} = \Vert \widetilde{\omega} \Vert_{\infty} = \frac{1}{4} \sup_{\widetilde{z} \in \DD}
\lv \widetilde{\xi} (\widetilde{z}) \rv \, \lb 1 -\lv \widetilde{z} \rv^2\rb^2 \, .
$$
A projective structure is said to be {\it bounded}\, if it has finite $L^{\infty}$-norm. Another standard norm associated with
projective structures is $L^1$-norm $\Vert \omega \Vert_{1}$ of $\omega$. The $L^1$-norm of $\omega$
also has a simple geometric interpretation as the area of $\mc{L}$ with respect to the singular Euclidean metric
$\vert \xi (z) \vert \vert dz\vert^2$. The $L^1$-norm is widely used in Teichm\"uller theory because
of its natural behavior under conformal maps though the its intrinsic meaning is
less clear than the finiteness of the $L^{\infty}$-norm. However, uniformizable projective structures verify the following:

\vspace{0.2cm}

\noindent {\bf Corollary B}.
{\sl Assume that $\mc{L}$ has finite hyperbolic area (e.g. it is a compact Riemann surface with finitely many punctures).
	If $\Ps$ is a uniformizable projective structure on $\mc{L}$ then we have
	$$
	\max \lcb \Vert \omega \Vert_{\infty} \, , \, \Vert \omega \Vert_1 \rcb < \infty \, ,
	$$
	where $\omega$ stands for the Schwarzian differential associated with $\Ps$. In other words, $\Ps$ is both
	bounded and of finite Euclidean area.}

\vspace{0.1cm}

Finally, the bounded character of $\Ps$ allows us to resort to previously known
results on {\it covering projective structures}\, to obtain
further insight on uniformizable ones, provided that we are dealing with a compact
Riemann surface (orbifold). More precisely results in \cite{shiga} and \cite{kra3}
imply the following:

\vspace{0.2cm}

\noindent {\bf Theorem C}.
{\sl Let $\mc{L}$ denote a Riemann surface with finite area and denote by $\mathcal{U} (\mc{L})$ the
	set of uniformizable projective structures on $\mc{L}$. Then the following holds:
	\begin{itemize}
		\item[(1)] The interior of $\mathcal{U} (\mc{L})$ coincides with quasi-conformal
		deformations of the canonical projective structure on $\mc{L}$. Equivalently,
		it coincides with the Bers embedding, centered at $\mc{L}$ (with classes of elliptic elements fixed), of
		the Teichm\"uller space of $\mc{L}$.

		\item[(2)] A projective structure $\Ps \in \mathcal{U} (\mc{L})$ is an isolated point of $\mathcal{U} (\mc{L})$
		if and only if the monodromy group of $\Ps$ contains no accidental parabolics and either the discontinuity domain
		of this monodromy group is connected or the quotients of the remaining components represent only thrice punctured
		spheres.
	\end{itemize}
}

To close this introduction, it is now convenient to explain the structure of our approach to the above theorems
together with their connections to some previous works. Since the definition
of uniformizable projective structure involves one-to-one - i.e. injective - developing maps, it is probably
convenient to begin by recalling that a
classical result of Kraus \cite{kraus} states that a complex projective structure on a hyperbolic Riemann surface
giving rise to an injective (i.e. one-to-one) developing map on the disc must have $L^{\infty}$-norm bounded by $3/2$
(see below for details).
As a partial converse, Nehari \cite{nehari} showed
that any projective structure whose $L^{\infty}$-norm is bounded by $1/2$ gives rise to an injective
developing map from the disc to $\CC$. Naturally both theorems essentially belong to the theory of univalent
functions on the unit disc $\DD$.

Compared to \cite{kraus}, \cite{nehari},
there is, however, a crucial difference which stems from the fact
that our developing maps are assumed to be injective on
the so-called {\it monodromy covering} rather than on the {\it universal covering}. As will soon be clear,
when dealing with injective developing maps, it is the monodromy covering, rather than the universal one, that stands out
as the most natural domain of definition.
Replacing the universal covering by the monodromy one, however,
makes the general theory of univalent functions on the disc a less effective tool
since the monodromy covering may be a very non-trivial quotient of the disc.
Incidentally, the bound of $3/2$ found by Kraus is no longer valid for uniformizable projective structures. In fact,
it is easy to adapt the construction in Section~\ref{sec:Cov_PS} to obtain examples of uniformizable
projective structures with arbitrarily large $L^{\infty}$ norm.

Along similar
lines, there is also a considerable gap between the two sets of projective, namely uniformizable projective structures
and projective structures whose developing map is injective on the universal covering. Indeed, the latter
immediately rules out the existence of elliptic elements in the monodromy group
hence imposing serious constraints on the projective structures in question, cf. \cite{annalspaper}.
Ruling out elliptic monodromy also excludes several cases of interest, beginning with the simplest possible case
which is provided by triangular groups and Schwarz automorphic functions. Clearly a triangular group
is the image of the fundamental group of the thrice punctured sphere by a homomorphism with rather large kernel.
In particular, the (automorphic) developing maps arising from the naturally associated projective structure are
injective on the monodromy covering but very far from this if considered on the universal covering.

In view of the above mentioned issues,
our approach to Theorems~A and~C will differ considerably from the methods in \cite{kraus}, \cite{nehari}.

On the other hand, another much studied class of (singular) projective structures on compact Riemann surfaces
are the so-called {\it bounded covering projective structures}, see for example \cite{kra2}, \cite{kra3}, and \cite{shiga}.
Whereas it is easy to check that a uniformizable projective structure is of covering type (cf. Lemma~\ref{lm:UPS_to_CPS}),
it is less clear whether or not the initial structure must be {\it bounded}. It is this boundedness issue that really prevents us
from obtaining Theorem~C straightly out of the results obtained by I. Kra and H. Shiga. We are then led to first work out
the general discussion leading to Theorem~A and Corollary~B.

Finally in the proof of Theorem~A, the local type of analysis needed to conclude, for example, the bounded
nature of the projective structure in question will take us close to an interesting conjecture put forward
by B. Elsner in \cite{elsner}, cf. Proposition~\ref{Key_proposition_oninjectivedevelopingmaps}.
As a matter of fact, Proposition~\ref{Key_proposition_oninjectivedevelopingmaps} is more general than what is strictly necessary
for this paper and, in fact, its use in the Theorem~A can be avoided. However, the ideas introduced
to establish Proposition~\ref{Key_proposition_oninjectivedevelopingmaps} seem
fit to deal with more general situations. Indeed, whereas we have not tried to fully push forward the method
to find out how much insight it can provide in Elsner's question, we believe that a serious attempt in this direction
would be worth a shot. Also, as another example of application of our ideas, we state Proposition~\ref{prp:inject_ess_sing} which appears naturally
in some related questions and whose proof can straightforwardly be obtained from our discussion.

As to the structure of this paper, it should be said that some effort was made to make it as self-contained as possible.
In particular, we have included
proofs of some lemmas that may be regarded as well-known to experts in projective
structures and related Teichmuller theory. The present paper, however, is likely to be of interest for
experts in differential equations and integral systems by virtue of the connection between uniformizable
projective structures and (generalized) Halphen
systems, to mention just one example. A
significant part of the latter community being less familiar with general Teichmuller theory
and with the role of quadratic differentials in it, however, we feel that our choice might be justified.

In closing this introduction, let us mention that another (potential) application of these ideas appears in the
context of foliated projective structures and it has some points of contact with the recent preprint
\cite{deroinguillot}. Whereas the paper \cite{deroinguillot} discusses the existence of foliated projective
structures on a (singular) holomorphic foliation on a complex surface, it would also be interesting to investigate
how far leaves (not necessarily all of them) of (singular) holomorphic foliations can be equipped
with {\it singular uniformizable projective structures}. This question was pointed out by B. Deroin to the second
author longtime ago. It might be viewed as a natural generalization of
the study of semicomplete vector fields, i.e., of vector fields with single-valued solutions.
Since the existence of uniformizable projective
structures on Riemann surfaces is a much more common phenomenon than the existence of semicomplete vector fields,
it is reasonable to wonder that the space of foliations admitting the singular uniformizable foliated projective
structures is singnificantly larger than the set of foliations tangent to a semicomplete vector field
(see \cite{adolfoandI} for the classification of the latter on K\"ahler surfaces). The upshot here is that
(singular) uniformizable foliated projective structures are still capable of providing some fine control on the Riemann mapping
uniformizing individual leaves. In other words, the set of foliations admitting
singular uniformizable foliated projective structures
may contain numerous interesting examples of foliations for which ``simultaneous uniformization problems'' can be
discussed in depth. Clearly, the problem of ``simultaneous uniformization for foliations'' has basic intrinsic importance and,
among their many applications, we single out the Ergodic theory of foliations as developed by
Sibony, Fornaess, Dinh, and Nguy\^en, see for example the nice survey \cite{vietanh} and
references therein.

\medskip
\noindent
\textbf{Acknowledgment.} This research was partially supported by CIMI through the project
``Complex dynamics of group actions, Halphen and Painlev\'e systems''.
A. Elshafei was supported by the FCT-grant PD/BD/143019/2018.
A. Elshafei and H. Reis were also
partially supported by CMUP, which is financed by national funds through FCT – Funda\c{c}\~ao para a Ci\^encia e Tecnologia, I.P.,
under the project UIDB/00144/2020. H. Reis was also supporter by FCT through the project ``Means and Extremes in Dynamical Systems''
with reference PTDC/MAT-PUR/4048/2021.


\section{Basic constructions and uniformizable projective structures}\label{sec:Basic_Schwarz}

In what follows, by a Fuchsian group (resp. Kleinian group) it is meant a discrete subgroup of
$\PSLR$ (resp. $\PSLC$). In other words, both Fuchsian and Kleinian groups are allowed to contain
elliptic elements of finite order.

The quotient of the hyperbolic disc by a Fuchsian group $\Gamma$ is not necessarily a (hyperbolic) Riemann
surface but rather a {\it Riemann surface orbifold}. Naturally a point of the disc $\DD$ that is fixed
by a non-trivial elliptic element of $\Gamma$, necessarily of finite order, gives rise to a singular point
in the quotient space $\DD /\Gamma$. In particular, the singular points of a Riemann surface orbifold must
form a discrete set. Moreover, if $p \in \DD /\Gamma$ is one of the singular points, then the local structure
of $\DD /\Gamma$ around $p$ is equivalent to the quotient of a small disc in $\CC$ by a finite group of
rotations.

Unless otherwise mentioned, all Riemann surfaces and/or Riemann surface orbifolds considered in this
paper are of {\it hyperbolic nature}, namely given by the quotient of $\DD$ by a Fuchsian group.
Let then $\mc{L}$ stand for a Riemann surface orbifold. A {\it singular projective structure}\, on $\mc{L}$
consists of the following data
\begin{itemize}
\item[(1)] A discrete set $\Upsilon \subseteq \mc{L}$ containing, in particular, all singular points
of $\mc{L}$.

\item[(2)] A projective structure on the Riemann surface $\mc{L} \setminus \Upsilon$. In other words,
the Riemann surface $\mc{L} \setminus \Upsilon$ is equipped with an atlas all of whose change of
coordinates coincide with suitable restrictions of M\"obius automorphisms of $\CC \PP^1$.
\end{itemize}
The points of $\Upsilon$ are thought of as the singular points of the projective structure in question.
In the sequel, we will also use the phrase {\it projective surface}\, to refer to a Riemann
surface equipped with a non-singular projective structure.

A singular projective structure on $\mc{L}$ gives rise to a homomorphism (monodromy representation)
$\mu$ from the fundamental group $\pi_1 (\mc{L}\setminus \Upsilon)$
to $\PSLC$ along with a developing map $\Dv$
from the universal covering of $\mc{L}$ to $\CC\PP^1$. Furthermore, the pair $(\Dv,\mu)$
satisfies the following equivariance relation
\begin{equation}\label{eq:dev-mon-rel}
	\Dv([c].p)=\mu (c).\Dv(p) \ ,
\end{equation}
for every $p$ in the universal covering of $\mc{L} \setminus \Upsilon$, $c \in \pi_{1}(\mc{L}\setminus \Upsilon)$
and where $[c]$ stands for the covering automorphism induced by $c$. However,
the homomorphism $\mu$ need not be injective and, in fact, there is basically no restriction on $\mu$ as
shown in \cite{annalspaper}. Hence the kernel ${\rm Ker}\, (\mu)$ is in general far from trivial.
Nonetheless, a direct inspection in the standard construction of the developing map, makes it clear
that such a map can be defined not only on the universal covering of $\mc{L}$ but on
{\it every covering $\widetilde{\mc{L}}$}\, such that $\pi_1 (\widetilde{\mc{L}})
\subset {\rm Ker}\, (\mu)$. Furthermore the resulting developing map still satisfies Equation~(\ref{eq:dev-mon-rel}).
In particular, the regular covering $\mc{L}_{\mu}$ associated with ${\rm Ker}\, (\mu)$ is the {\it smallest
covering of $\mc{L}$} on which a developing map is defined. The covering $\mc{L}_{\mu}$ will be called
the {\it monodromy covering}\, of $\mc{L}$ and the corresponding developing map
$\Dv_{\mu} : \mc{L}_{\mu} \rightarrow \CC\PP^1$ will be called the {\it monodromy-developing map}.

\begin{df}\label{def:uniformizable}
With the above notations, the singular projective structure $\Ps$ on $\mc{L}$ is said to be uniformizable if the developing
map $\Dv_{\mu}$ is injective on $\mc{L}_{\mu}$.
\end{df}

It follows from Definition~\ref{def:uniformizable} that, if $\Ps$ is a uniformizable singular projective structure
on $\mc{L}$, then the restriction of $\Ps$ to any open subset $U$ of $\mc{L}$ is still uniformizable on $U$.
The lemma below provides an alternative characterization of uniformizable singular projective structures.

\begin{lm}\label{lm:uniformizable_vs_quotient}
The singular projective structure $\Ps$ is uniformizable if and only if the projective
surface $\mc{L} \setminus \Upsilon$ is
isomorphic to the quotient of an open (invariant) subset $\mathcal{U}$ of $\CC \PP^1$
by a Kleinian group.
\end{lm}

\begin{proof}
Owing to the equivariance relation~(\ref{eq:dev-mon-rel}), the image $\mathcal{U} =
\Dv_{\mu} (\mc{L}_{\mu}) \subseteq \CC\PP^1$ of $\mc{L}_{\mu}$ by $\Dv_{\mu}$ is an open set
of $\CC\PP^1$ which is invariant by $\mu (\pi_1(\mc{L} \setminus \Upsilon))$, the
subgroup of $\PSLC$ obtained as image of the monodromy representation.
Now, if $\Dv_{\mu}: \mc{L}_{\mu} \to \CC\PP^1$ is injective,
then $\mathcal{U}$ is diffeomorphic to $\mc{L}_{\mu}$ which, in turn, induces a diffeomorphism
between the quotients of $\mc{L}_{\mu}$ and of $\mathcal{U}$ by the corresponding actions of
$\pi_1(\mc{L} \setminus \Upsilon)$. Hence, $\mc{L}$ is diffeomorphic to the
quotient of $\mathcal{U}$ by $\mu (\pi_1(\mc{L} \setminus \Upsilon))$. In particular,
$\mu (\pi_1(\mc{L} \setminus \Upsilon))$ must be a discrete subgroup of $\PSLC$. Also,
the projective structure $\Ps$ coincides with the evident projective structure on the
quotient of $\mathcal{U}$.

The converse, follows directly from the standard construction of developing maps.
\end{proof}

Since our purpose is to understand uniformizable singular projective structures,
it is convenient to parameterize
the space of all projective structures on a given Riemann surface orbifold. A standard parameterization
of these structures can be obtained by fixing a particular projective structure and then comparing all the other
projective structures with the fixed one.
This comparison between projective structures is quantified by the Schwarzian differential
whose definition is summarized as follows. First recall that the {\it Schwarzian derivative}\, of a holomorphic
function $f$ defined on an open set of $\CC$ is given by
\begin{equation}\label{eq:Schawrzian_def}
\Sch_{z}(f) = \lb \frac{f''(z)}{f'(z)} \rb' - \frac{1}{2} \lb \frac{f''(z)}{f'(z)} \rb^{2} \ .
\end{equation}
It is straightforward to check that the Schwarzian derivative satisfies the following invariance properties:
\begin{equation}\label{eq:Schawrzian_invar}
\Sch_{z}(f\circ\gamma)=\Sch_{\gamma(z)}(f)\gamma'(z)^{2} \qquad \text{and} \qquad \Sch_{z}(\gamma\circ f)=\Sch_{z}(f)
\end{equation}
for every M\"obius transformation $\gamma$ and every holomorphic function $f$.

Let then $\Ps_0$ be a fixed singular projective structure on $\mc{L}$ which is given by the atlas
$\mathcal{A}_0 = \{(U_i, \psi_i)\}$.
Consider also a second singular projective structure $\Ps$ on $\mc{L}$ whose atlas is denoted
by $\mathcal{A} = \{(V_j, \varphi_j)\}$. For every pair $(i,j)$ for which $U_i \cap V_i \ne \emptyset$, we consider the
composition map
\[
\varphi_j \circ \psi_i^{-1} :  \psi_i(U_i \cap V_j) \subset \CC \longrightarrow \CC
\]
which corresponds to a holomorphic function defined on a open set of $\CC$ with standard coordinate~$z$.
The Schwarzian derivative of the function $\varphi_j \circ \psi_i^{-1}$ will be denoted by $\xi_{ij}$
so that we have $\xi_{ij} = \Sch_{z} ( \varphi_j \circ \psi_i^{-1} )$.
Owing to the invariance properties of the Schwarzian derivative~(\ref{eq:Schawrzian_invar}), the resulting
collection of locally defined functions $\{ \xi_{ij} \}$ patch together as a global holomorphic
quadratic differential on $\mc{L} \setminus \Upsilon$ (or more generally on the complement of the union of the
singular sets of $\Ps_0$ and $\Ps$).
This quadratic differential will be denoted by $\omega = \xi (z) dz^2$ and referred to as the {\it Schwarzian differential}.

Alternatively, the Schwarzian differential can be defined as follows. Denote by $\widetilde{\mc{L}}$
the universal covering of $\mc{L} \setminus \Upsilon$ and let $\Dv$ be the resulting
developing map of $\Ps$. The atlas $\mathcal{A}_0 = \{(U_i, \psi_i)\}$ for $\Ps_0$
naturally induces an atlas
$\widetilde{\mathcal{A}}_0 = \{(\widetilde{U}_i, \widetilde{\psi}_i)\}$ for $\widetilde{\mc{L}}$.
Thus, for every open set
$\widetilde{U}_i$ of $\widetilde{\mathcal{A}}_0$ we may consider the map
$\Dv \circ \widetilde{\psi}^{-1}_i : \widetilde{\psi}_i(U_i) \subset \CC  \to \CC$.
Note that this map is holomorphic and defined on an open set of $\CC$. Also, it does not depend on
the choice of the lift of $U_i$.
The Schwarzian derivatives of these locally defined maps
can thus be taken and the resulting collection of locally defined functions again patch together as
a quadratic differential on $\mc{L} \setminus \Upsilon$ coinciding with the Schwarzian differential
$\omega$ as previously defined.

An immediate consequence of the above construction is that the Schwarzian differential $\omega$ vanishes
identically if and only if the projective structures $\Ps_0$ and $\Ps$ coincide. A converse also holds, if
$\omega'$ is a quadratic differential on $\mc{L} \setminus \Upsilon$, then by considering local solutions
for the Schwarzian differential equations in the coordinates of $\Ps_0$ (see for example Section~\ref{sec:mon_punc}),
a new singular projective structure $\Ps'$ on $\mc{L}$ is defined. In addition, $\omega'$ coincides
with the Schwarzian differential of $\Ps'$. Summarizing,
the choice of the initial $\Ps_0$
allows us to identify the set of (singular) projective structures on $\mc{L}$ with the space
of quadratic differentials on $\mc{L}$ (with suitable singular sets that the reader can easily work out).

To begin the approach to Theorem~A, let us consider the behavior of a uniformizable singular projective structure
$\Ps$ on $\mc{L}$ around a singular point $p \in \Upsilon \subset \mc{L}$.
For the time being we make no distinction between the cases where the point $p \in \mc{L}$ is a regular (smooth) point
or an orbifold
type singular point. Consider a small neighborhood $\mathcal{U} \subset \mc{L}$ of $p$. The restriction $\Ps|_{\mathcal{U}}$
of $\Ps$ to $\mathcal{U}$ is therefore uniformizable and we can assume without loss of generality that the
monodromy group associated to $\Ps|_{\mathcal{U}}$ has a unique generator $\gamma$ arising from
a small loop around $p$. Since $\mu (\pi_{1} (\mathcal{U} \setminus \{p\}))$ is a discrete
subgroup of $\PSLC$ (cf. Lemma~\ref{lm:uniformizable_vs_quotient}), $\gamma$ must be of one of the following types
\begin{itemize}
\item[(1)] Hyperbolic.
\item[(2)] Elliptic with finite order $k$.
\item[(3)] Parabolic.
\end{itemize}
However, the proof of Lemma~\ref{lm:uniformizable_vs_quotient} shows a bit more. In fact, it shows that
for $p$ to be an orbifold point, it has to be identified with a fixed point of an elliptic element $h$ of finite order.
Therefore, $\mathcal{U} \setminus \{p\}$ must be
given as the quotient of $\DD$ by $h$. Furthermore, since our projective structure is uniformizable,
it follows from the proof of Lemma~\ref{lm:uniformizable_vs_quotient} that $\mathcal{U} \setminus \{p\}$ is also the quotient
of an open set of $\CC\PP^1$, or equivalently of $\DD$, by the monodromy group
$\mu (\pi_{1} (\mathcal{U} \setminus \{p\}))$. In particular, $\gamma$ has a fixed point at $p$.
More precisely, $\gamma$ is the element defining the orbifold so that it is also
elliptic of order~$k$. Conversely, assume that the generator $\gamma$ of the monodromy group
is not of elliptic type. It then follows that the action of $\mu (\pi_{1} (\mathcal{U} \setminus \{p\}))$
$\DD$ is properly
discontinuous so that $\mathcal{U}$ must be smooth at $p$. We have then proved the following.

\begin{lm}\label{prp:condition_orbifold}
The point $p$ is an orbifold point if and only if the local monodromy map around $p$ is elliptic of finite order.\qed
\end{lm}

Let us close this section by showing that the monodromy map around a singular point cannot
be of hyperbolic type. This result is probably known to
experts and the argument given below is a minor adaptation from \cite{Guillot_Thesis}.

\begin{lm}\label{lm:UPS_disc*_not-hyp}
Let $\Gamma = \langle \gamma \rangle \subset \PSLC$ be the monodromy group of a
uniformizable projective structure $\Ps$ on the punctured disc $\DD^{\ast}$.
Then $\gamma$ is either an elliptic or a parabolic transformation.
\end{lm}

\begin{proof}
Assume aiming at a contradiction that $\gamma$ is hyperbolic. In particular the monodromy
covering of $\DD^{\ast}$ coincides with its universal covering $\DD$ since $\gamma$ has infinite order.
Since $\Ps$ is uniformizable, there follows that
the developing map $\Dv : \DD \longrightarrow \CC\PP^1$ is injective. Thus, $\Dv \, (\DD)\subset\CC\PP^1$
is conformally equivalent to the unit disc. Moreover, since $\Gamma$ acts freely on $\DD$ (as deck-transformations),
the equivariance of the developing map implies that the action of $\Gamma$ on $\CC\PP^1$ (given by the dynamics of $\gamma$)
leaves invariant $\Dv \, (\DD)$ and, in addition, is free when restricted to the $\Gamma$-invariant set
$\Dv \, (\DD)$. Hence, the two fixed points of the hyperbolic
element $\gamma$ must lie in the boundary of $\Dv(\DD)$.

Let us denote by $\mc{C}$ the quotient of $\Dv \, (\DD)$ by $\Gamma$, $\Dv \, (\DD) / \Gamma$, which
is a hyperbolic cylinder. Since $\Dv$ is injective and equivariant with respect to the two actions
of $\Gamma$, it induces a diffeomorphism $\Upsilon : \DD^{\ast}
\longrightarrow \mc{C}$. In particular, if $c$ is a loop representing the generator $[c]$
of the fundamental group of $\DD^{\ast}$, then $\Upsilon (c)$ is a loop in $\mc{C}$ representing
the generator $[\Upsilon (c)]$ of the fundamental group of $\mc{C}$.

Let $\DD^{\ast}$ be equipped with the Poincar\'e metric ${\it m}_{\DD}$ induced from its universal
covering $\DD$. Clearly, we can choose loops $c \in \DD^{\ast}$ winding once around $0 \in \DD$ and
having arbitrarily small length with respect to ${\it m}_{\DD}$. Similarly,
let $\mc{C}$ be endowed with its Poincar\'e metric ${\it m}_{\mc{C}}$ obtained by taking the quotient of the
disc by the action of the hyperbolic map $\gamma$. The length of $\Upsilon (c)$ with respect to ${\it m}_{\mc{C}}$ must
be bounded from below by a strictly positive number $\varepsilon
> 0$ since $\mc{C}$ is a hyperbolic cylinder.
However, since $\Upsilon$ is a holomorphic map, the pull-back $\Upsilon^{\ast}{\it m}_{\mc{C}}$
of the Poincar\'e metric ${\it m}_{\mc{C}}$ on $\mc{C}$
by $\Upsilon$ is a (hyperbolic) metric
on $\DD^{\ast}$ that is bounded by the previously introduced Poincar\'e metric ${\it m}_{\DD}$ on $\DD^{\ast}$
(cf. \cite{Ahlfors_book}). The desired contradiction follows at once from observing that the lenght of $c$
with respect to ${\it m}_{\DD}$ can be made arbitrarily small whereas its length with respect to
$\Upsilon^{\ast}{\it m}_{\mc{C}}$ is bounded from
below by $\varepsilon >0$.
\end{proof}

%
%


\section{Uniformizable projective structures with trivial monodromy}\label{sec:mon_punc}

According to Lemma~\ref{lm:UPS_disc*_not-hyp}, the monodromy map $\gamma$ of a uniformizable projective
structure $\Ps$ around a singular point $p \in \mc{L}$ is either an elliptic or a parabolic transformation.
These two cases will be discussed separately. This section is devoted to the elliptic case.

For $\Ps$ and $p$ as above, we assume in the sequel that the monodromy map $\gamma$ is
elliptic (necessarily having finite order~$k$). Lemma~\ref{prp:condition_orbifold} then ensures that
$p$ is an orbifold point and that a neighborhood $\mc{U}\subset\mc{L}$ of $p$, is isomorphic to the quotient of $(\CC ,0)$
by a rotation of angle $2\pi /k$. Let us first point out that this general situation can straightforwardly be reduced to the case
of trivial monodromy and smooth point. Indeed, note that the ramified $k$-sheet covering of a
neighborhood of $p$ in $\mc{L}$ is identified with a neighborhood $U$ of $0 \in \CC$. In particular,
the singular projective structure $\Ps$ can be lifted to a projective structure $\Ps^{(k)}$
on $U$ which is regular away from $0 \in U \subset \CC$. Also, by construction, the monodromy of
$\Ps^{(k)}$ is trivial so that the monodromy-developing map of $\Ps^{(k)}$ is defined on
$U \setminus \{ 0\}$ and hence coincides with the monodromy-developing map of $\Ps$ on a punctured
neighborhood of $p \in \mc{L}$. In particular, $\Ps^{(k)}$ is uniformizable on $U \setminus \{ 0\}$.

In view of the preceding, let us focus on uniformizable (singular) projective structures on the disc
$\DD$ which are regular on $\DD^{\ast}$ and whose monodromy is trivial. We choose as initial projective
structure $\Ps_0$ the (regular, uniformizable) projective structure induced by the standard coordinate $z \in \CC$
(making use of the evident inclusion $\DD^* \subset \CC$).
Given a regular uniformizable projective structure $\Ps$ on $\DD^{\ast}$ with trivial monodromy,
the coefficient $\xi$ of the Schwarzian
differential coincides with the Schwarzian derivative of the monodromy developing map $\Dv_\mu$ of $\Ps$ written
in the coordinate $z$. Clearly $\Dv_\mu$ is well defined on $\DD^{\ast}$ and satisfies
\begin{equation}\label{eq:phi}
\xi (z) = \Sch_{z} \Dv_\mu = \lb \frac{\Dv_\mu''(z)}{\Dv_\mu'(z)} \rb' - \frac{1}{2} \lb
\frac{\Dv_\mu''(z)}{\Dv_\mu'(z)} \rb^{2} \, .
\end{equation}
Moreover, we have the following:

\begin{lm}\label{lm:UPS_phi_triv_mon}
Let $\Ps$ be a uniformizable regular projective structure on $\DD^{\ast}$ and assume that its monodromy
is trivial. Then $0 \in \DD$ is not an
essential singularity for $\Dv_\mu$. More precisely, one of the following holds:
\begin{itemize}
\item[(1)] $\Dv_\mu$ is meromorphic on $\DD$ with a simple pole at $0\in\DD$.
\item[(2)] $\Dv_\mu$ is holomorphic on $\DD$ and $\Dv_\mu'(0)\neq 0$.
\end{itemize}
\end{lm}

\begin{proof}
Since the monodromy is trivial, $\Dv_\mu$ is by construction well defined and holomorphic on $\DD^{\ast}$.
Furthermore, $\Dv_\mu : \DD^{\ast} \rightarrow \CC$ is injective since $\Ps$ is uniformizable. In
particular, it follows from Picard Theorem that $0 \in \DD$ is not an essential singularity of $\Dv_\mu$.
Hence, $\Dv_\mu$ is (at worst) meromorphic on $\DD$ and can be written as
\begin{align*}
\Dv_\mu(z) = az^{l} + {\rm h.o.t.}
\end{align*}
for some $l \in \ZZ$ and $a \in \CC^{\ast}$. We can assume that $l \neq 0$, otherwise the constant term
``$a$'' can be eliminated by adding a translation to $\Dv_\mu$ which would yield an equivalent developing
map with strictly positive order at $0 \in \CC$. Next, we have:

\vspace{0.2cm}

\noindent {\it Claim}: The map $z \longmapsto az^{l}$, $l \neq 0$, is injective.

\begin{proof}[Proof of the claim]
Consider the map $\Lambda(z) = \lambda z$ for some $\lambda \in \CC^{\ast}$, which is clearly injective.
Since, $\Dv_\mu$ is injective on $\DD^{\ast}$, the conjugate $\Lambda^{-1} \Dv_\mu \Lambda$ is
also injective on $\DD^{\ast}$ so long $\lv \lambda \rv \leq 1$. However, we have
\begin{align*}
\Lambda^{-1} \Dv_\mu \Lambda(z) = \lambda^{-1} \Dv_\mu (\lambda z) =
\lambda^{-1+l} \lsb az^l\lb 1 + \lambda O(z) \rb \rsb \ .
\end{align*}
Hence for $\lambda \in \CC^{\ast}$ with $\lv \lambda \rv \leq 1$,
all the maps $z \longmapsto az^{l} (1 + \lambda O(z))$ are injective on $\DD^{\ast}$. Hence so is the map
$z \longmapsto az^{l}$ as a non-constant uniform limit of injective maps.
The claim is proved.
\end{proof}

Since $l \ne 0$, there follows from the above claim that $l$ is either $1$ or $-1$.
When $l = -1$, $\Dv_\mu$ is meromorphic with a simple pole at $0 \in \DD$. In turn, if $l = 1$,
then $\Dv_\mu$ is holomorphic and $\Dv_\mu'(0) \neq 0$. The lemma is proved.
\end{proof}

Proposition~\ref{prp:holomorphic_extension_trivial_mono} is a useful consequence of
Lemma~\ref{lm:UPS_phi_triv_mon}.

\begin{prp}\label{prp:holomorphic_extension_trivial_mono}
Let $\Ps$ be a projective structure as in Lemma~\ref{lm:UPS_phi_triv_mon}.
Then the coefficient $\xi$ of its Schwarzian differential
possesses a holomorphic extension to the entire disc $\DD$.
\end{prp}

\begin{proof}
Owing to Lemma~\ref{lm:UPS_phi_triv_mon} and up to adding a suitable translation, the monodromy-developing
map $\Dv_{\mu}$ of $\Ps$ can be written in Laurent series as
\begin{align*}
\Dv_{\mu}(z) = a_{-1} z^{-1} + a_{1} z + a_2 z^2 + \cdots \ .
\end{align*}
for some constants $a_i \in \CC$ where at least one between $a_{-1}$ and $a_1$ is non-zero. In turn,
the coefficient $\xi$ of the Schwarzian differential is related to $\Dv_{\mu}$ by means of
Formula~(\ref{eq:phi}). In particular, when $\Dv_{\mu}$ is holomorphic at $0 \in \CC$, it follows at once
from Formula~(\ref{eq:phi}) that $\xi$ is holomorphic on all of $\DD$ as well.
It is therefore sufficient to discuss the case where $a_{-1} \ne 0$. Then a direct computation
provides
\[
\frac{\Dv_\mu''(z)}{\Dv_\mu'(z)} = \frac{2a_{-1} z^{-3} (1 + z^3 A_2(z))}{-a_{-1} z^{-2} (1 + z^2 A_1(z))}
= -\frac{2}{z}(1 + z^2A(z))
\]
for some holomorphic functions $A, A_1$ and $A_2$. It then follows that
\[
\left( \frac{\Dv_\mu''(z)}{\Dv_\mu'(z)} \right)' = \frac{2}{z^2} + G(z) \qquad \text{and} \qquad
\left( \frac{\Dv_\mu''(z)}{\Dv_\mu'(z)} \right)^2 = \frac{4}{z^2} + H(z)
\]
for some holomorphic functions $G$ and $H$. Now, Formula~(\ref{eq:phi}) ensures
that $\xi$ is holomorphic at the origin.
\end{proof}

To close this section, let us return to the initial problem of describing a uniformizable
singular projective structure $\Ps$ on a neighborhood of the orbifold point $p \in \mc{L}$.
Recall that the local structure of $\mc{L}$ around $p$ is isomorphic to the quotient
of the disc $\DD$ by a rotation of angle $2\pi /k$, $k \geq 2$. In particular, $\mc{L}$
can be equipped with the local ``orbifold-type'' coordinate $y$ defined by $y = \sqrt[k]{z} = \eta (z)$
where $z$ is the standard coordinate on $\DD \subset \CC$.

As previously seen, $\Ps$ lifts to a uniformizable (regular) projective structure $\Ps^{(k)}$ on $\DD^*$ with
trivial monodromy and
whose Schwarzian differential will be denoted by $\xi^{(k)} (y) dy^2$
(Proposition~\ref{prp:holomorphic_extension_trivial_mono}).
Denoting by $\xi_{\rm orb}(y) dy^2$ the Schwarzian differential of $\Ps$ in the singular (orbifold-type) coordinate $y$,
it follows that
\begin{equation}
\xi_{\rm orb} (y) = \xi^{(k)} (\eta^{-1}(y)) \lsb(\eta^{-1}(y))'\rsb^2 = k^2 y^{2k-2} \xi^{(k)}(y^k) \, .
\label{Schwarzian_orbifoldcoordinate}
\end{equation}

To summarize the discussion in this section, we state:

\begin{prp}\label{prp:generalorbifoldform_ellipticmono}
Consider a uniformizable singular projective structure $\Ps$ on the orbifold given by the quotient
of $\DD$ by a rotation of angle $2\pi /k$, $k \geq 2$. Then all of the following hold:
\begin{itemize}
	\item The monodromy of $\Ps$ is generated by an elliptic element of order~$k$.
	
	\item In the orbifold coordinate $y = \sqrt[k]{z}$, the Schwarzian differential of $\Ps$
	is given by Formula~(\ref{Schwarzian_orbifoldcoordinate}), where $\xi^{(k)}$ is holomorphic
	on all of $\DD$.
	
	\item Both the $L^{\infty}$ and the $L^1$ norms of $\Ps$ are locally bounded around the singular point.
\end{itemize}
\mbox{ }\qed
\end{prp}


\section{The parabolic case}\label{sec:mon_punc}

It remains to discuss the nature of a uniformizable projective structure $\Ps$ around a singular
point $p \in \mc{L}$ in the case where the local monodromy is a parabolic map.
Since Lemma~\ref{prp:condition_orbifold} ensures that $\mc{L}$ is smooth at $p$, we can once and for
all assume that $\Ps$ induces a uniformizable regular projective structure - still denoted by $\Ps$ - on
the punctured disc $\DD^{\ast} \subset \CC$ equipped with the standard coordinate $z \in \CC$.
The monodromy of $\Ps$ on $\DD^{\ast}$ is a parabolic map denoted by $\gamma$.

Taking advantage of the coordinate $z$ defined on all of $\DD$, the Schwarzian differential can be written as
$\xi (z) \, dz^2$. Also, the monodromy covering of $\DD^{\ast}$ coincides with its universal covering
and the corresponding developing map $\Dv_{\mu}$ locally satisfies the differential equation
\begin{equation}\label{eq:Schwarz_eqn1}
\xi (z) = \Sch_{z} \Dv_\mu \, .
\end{equation}
It is well known (cf.~\cite{Ince}) that two solutions of Equation~(\ref{eq:Schwarz_eqn1}) differ by
post-composition with a M\"obius transformation. Furthermore, any solution of this equation
is given as the quotient of two suitable linearly independent solutions of the linear ordinary
differential equation
\begin{equation}\label{eq:Schwarz_eqn2}
u''(z)+\frac{1}{2}\xi (z)u(z)=0 \ .
\end{equation}

Being a linear equation, the local solutions of~(\ref{eq:Schwarz_eqn2}) can be continued along paths. In particular,
the same holds for the solutions of Equation~(\ref{eq:Schwarz_eqn1}). Thus, every solution of~(\ref{eq:Schwarz_eqn1})
is globally defined on $\DD \setminus \RR_+$, where $\RR_+$ stands for the set of non-negative reals. Also, given
$x_0 \in \DD \cap \RR_+$ and
a local solution $\phi$ of~(\ref{eq:Schwarz_eqn1}),
consider the continuations of $\phi$ along the path $t \mapsto x_0 e^{ i t}$, $t \in (0,2\pi)$. This gives us two local solutions
around $x_0$, denoted respectively by $\phi_0$ and $\phi_1$ and, by construction, we have $\phi_1 = \gamma \circ \phi_0$.
However, since $\gamma$ is parabolic, we can assume without loss of generality that $\gamma (z)  = z + c$ for some
$c \in \CC^{\ast}$. Thus, we conclude that
$\lim_{t \rightarrow 0^+} \phi (x_0 e^{it}) - \lim_{t \rightarrow 2\pi^-} \phi (x_0 e^{it}) = c$.
Similar conclusions hold if $\RR_+$ is replaced by any semi-line issued from $0 \in \CC$. More precisely, given
$\theta \in [0, 2 \pi]$, let $l^{\theta}$ be defined by
\begin{equation}\label{sectors_with_ltheta}
l^{\theta} = \{ z \in \CC \; \, :  \; \; \; z=0 \; \; \; {\rm or} \; \; \; z=e^{a+i \theta} \; \; \; {\rm with} \; \; \; a \in \RR \} \, .
\end{equation}
Clearly for every $\theta$ as above, the solutions of~(\ref{eq:Schwarz_eqn1}) are globally defined on $\DD \setminus l^{\theta}$.
Furthermore, if $\phi$ is a local solution continued along a path $t \mapsto x_0 e^{i t}$, $t \in (\theta, \theta + 2\pi)$, then
we still have
\begin{equation}
\lim_{t \rightarrow \theta^+} \phi \lb x_0 e^{it}\rb \, \, \, - \lim_{t \rightarrow (\theta + 2\pi)^-} \phi \lb x_0 e^{it}\rb
= c \, .\label{difference_between_limits}
\end{equation}

%
%

Equation~(\ref{eq:Schwarz_eqn2}) is well known and, in general, it is not easy (possible) to find explicit solutions. Yet,
the standard method of ``variation of parameters'' yields the following:

\begin{lm}
\label{variation_of_constant}
Assume that $h : U \subset \CC \rightarrow \CC$ is a (non-zero) solution	of~(\ref{eq:Schwarz_eqn2}), where $U \subset \CC$ is a
simply connected domain. Then $\phi : U \rightarrow \CC$ defined by
\begin{equation}
\phi (z) = \int_{z_0}^z \frac{1}{h^2 (s)} ds + {\rm const} \label{firstformula_phi_outof_u}
\end{equation}
is a solution of Equation~(\ref{eq:Schwarz_eqn1}), where $z_0 \in U$ is a fixed base point and ${\rm const} \in \CC$.
\end{lm}

\begin{proof}
The standard method of ``variation of parameters'' indicates the existence of a solution $u(z)$ for Equation~(\ref{eq:Schwarz_eqn2})
having the form $u(z) = \phi (z) h(z)$ for some non-constant function $\phi$. Up to finding such a solution
$u(z)$, it follows from the preceding that $\phi$ is a solution of Equation~(\ref{eq:Schwarz_eqn1}).
On the other hand, $u = \phi h$ is solution of~(\ref{eq:Schwarz_eqn2})
if and only if $\phi$ verifies $h \phi'' + 2h'\phi' =0$ which is a linear first order equation on $\phi'$. From this, it
promptly follows that
$$
\phi'(z) = \frac{1}{h^2(z)}\, .
$$
Equation~(\ref{firstformula_phi_outof_u}) follows at once.
\end{proof}

Since any solution $\phi$ of~(\ref{eq:Schwarz_eqn1}) is obtained as the quotient of two suitable solutions
of~(\ref{eq:Schwarz_eqn2}), we can assume that $\phi$ as in~(\ref{firstformula_phi_outof_u})
also satisfies~(\ref{difference_between_limits}). In other words, the function
\begin{equation}\label{function_gofx}
g(z) = \int_{z_0}^z \frac{1}{h^2 (s)} ds + {\rm const} - \frac{c}{2\pi i}  \int_{z_0}^z \frac{1}{s}ds =
\phi (z) - \frac{c}{2\pi i} \ln (z/z_0)
\end{equation}
is holomorphic on the punctured disc $\DD^{\ast}$. In fact, the multivalued character of $\phi$ represented by
Equation~(\ref{difference_between_limits}) cancels the multivaluedness of the logarithm to ensure that the analytic
extension of $g$ over the loop $t \mapsto x_0 e^{ i t}$, $t \in [0,2\pi]$ is well defined.

Now recalling that $\Ps$ is uniformizable on $\DD^{\ast}$, we conclude that the extension $\Dv_{\mu}$ of $\phi$ given
by~(\ref{firstformula_phi_outof_u}) on the universal covering $\DD$ of $\DD^{\ast}$ gives an
injective map from $\DD$ to $\CC\PP^1$.
Here the reader is reminded that the monodromy covering of $\DD^{\ast}$ coincides with its universal covering
since the monodromy representation of $\Ps$ arises from a parabolic M\"obius transformation.

The remainder of this section is devoted to deriving specific information on the behavior
of the Schwarzian differential $\omega=\xi dz^2$
from the fact that $\Dv_{\mu} : \DD \rightarrow \CC\PP^1$ is an injective map. As a first evident consequence of
the fact that $\Dv_{\mu}$ is injective, we see that
the restriction of $\phi$ to any domain of the form $\DD \setminus l^{\theta}$ is injective as well.
Nonetheless, we have the following:

\begin{prp}\label{Key_proposition_oninjectivedevelopingmaps}
Let $g$ and $\phi$ be as above. If the restrictions of $\phi$ to domains of the form $\DD \setminus l^{\theta}$
are injective then $g$ cannot have an essential singular point at $0 \in \CC$.
\end{prp}

The proof of Proposition~\ref{Key_proposition_oninjectivedevelopingmaps} will be deferred to Section~\ref{sec:Appendix}.
The remainder of the present section is devoted to providing a detailed description of the Schwarzian differential
$\omega =\xi dz^2$ assuming throughout the sequel that $\Ps$ is uniformizable on $\DD^{\ast}$. First we have:

\begin{lm}
\label{xi_no-essentialsingularity}
Let $\Ps$ be a uniformizable projective structure on $\DD^{\ast}$ with parabolic monodromy. Then, its Schwarzian differential
$\omega =\xi dz^2$ cannot have an essential singularity at $0 \in \CC$.
\end{lm}

\begin{proof}
Keeping the preceding notation, consider the function $g$ defined by~(\ref{function_gofx}). In view of
Proposition~\ref{Key_proposition_oninjectivedevelopingmaps}, $g$ is meromorphic (possibly holomorphic) on $\DD$.
In particular, its derivative $g'$ is meromorphic on $\DD$ what, in turn, informs us that $h^2$ is meromorphic
as well. Though $h$ does not need to be meromorphic as the square root of a meromorphic function, it does have a well-defined
order at $0 \in \CC$. More precisely, there is $\kappa \in \RR$ such that
$$
\lim_{z \rightarrow 0} \vert h (z) \vert = O (\vert z \vert^{\kappa}) \, ,
$$
where $O (\vert z \vert^{\kappa})$ means that the left side is bounded - from below and by above - by suitable constant multiples
of $\vert z \vert^{\kappa}$. Also,
$\kappa$ is such that $\kappa - [\kappa] \in \{ 0,1/2\}$, where the brackets $[\, \cdots \,]$ stand for the integral
part, i.e., $\kappa$ is either an integer or half an integer. The same applies to its second derivative $h''$ which satisfies
$$
\lim_{z \rightarrow 0} \vert h'' (z) \vert = O (\vert z \vert^{\kappa'}) \,
$$
for a suitable $\kappa' \in \RR$. Equation~(\ref{eq:Schwarz_eqn2}) then implies that
$$
\lim_{z \rightarrow 0} \vert \xi (z) \vert =  O (\vert z \vert^{\kappa' - \kappa} ) \, .
$$
In view of Picard theorem (actually the more elementary Casorati-Weierstrass theorem), the last limit implies that $\xi$ cannot have
an essential singular point at $0 \in \CC$. The lemma is proved.
\end{proof}

Next, let us investigate what can be said about the order of $\phi$ at~$0 \in \CC$. For this, we let $\phi (z)=
g(z) + c \ln (z/z_0) /2\pi i$. Since $g$ is meromorphic, we can set
$$
g(z) = \sum_{k = k_0}^{\infty} a_k z^k
$$
for some $k_0 \in \ZZ$ and $a_{k_0} \neq 0$. Now, similar to Lemma~\ref{lm:UPS_phi_triv_mon}, the following holds:

\begin{lm}\label{k_0isatworst-1}
We have $k_0 \geq -1$ provided that $\Ps$ is uniformizable.
\end{lm}

\begin{proof}
Assume for a contradiction that $k_0 \leq -2$ and choose a determination of $\ln$ on the set $\DD \setminus l^{3\pi/2}$.
We will show that $\phi$ cannot be injective on $\DD \setminus l^{3\pi/2}$
hence contradicting the condition that $\Ps$ is uniformizable.

To begin, we have
\begin{eqnarray*}
\phi (z) & = &  z^{k_0} \lsb (a_{k_0} + a_{k_0+1} z + {\rm h.o.t.} ) + \frac{c}{2\pi i} z^{-k_0} \ln (z/z_0) \rsb \\
&  = & a_{k_0} z^{k_0}
+ z^{k_0+1} \lsb (a_{k_0+1}  + {\rm h.o.t.}) + \frac{c}{2\pi i} z^{-k_0- 1}\ln (z/z_0) \rsb \\
& = & a_{k_0} z^{k_0}  + z^{k_0+1} \Psi_1(z) \, ,
\end{eqnarray*}
where $\Psi_1$ is holomorphic on $\DD \setminus l^{3\pi/2}$.
Since $-k_0 -1\geq 1$ and $\lim_{z \rightarrow 0} \vert z  \ln z \vert =0$,
there follows that the function $z \mapsto \vert \Psi_1 (z) \vert$
is uniformly bounded by a constant $C_1$ on $\DD \setminus l^{3\pi/2}$. On the other hand,
$\phi'(z) = g'(z) \; + \; c/2\pi iz$ so that we similarly obtain
$$
\phi'(z) = k_0 a_{k_0} z^{k_0 -1} + z^{k_0} \Psi_2 (z)
$$	
where $\Psi_2$ is, in fact, holomorphic on all of $\DD$. In particular $\Psi_2$ is bounded by some constant $C_2$.
Also, both $\vert \phi' (\epsilon) \vert$ and $\vert \phi' (\lambda \epsilon) \vert$ are $O (\vert \epsilon^{k_0 -1} \vert)$,
where $\lambda$ is the $\vert k_0 \vert$-root of the unity with smallest real part
while $\epsilon > 0$ is sufficiently small.
Let $B_1$ and $B_2$ denote the balls of (same) radius~$\epsilon /2$ respectively centered at $\epsilon$ and at $\lambda \epsilon$.
Clearly these balls are disjoint and contained in $\DD \setminus l^{3\pi/2}$. Since $\phi$ is injective on $\DD \setminus l^{3\pi/2}$,
the contradiction proving Lemma~\ref{k_0isatworst-1} arises at once from
the following claim:

\noindent{\it Claim}. The intersection $\phi (B_1) \cap \phi (B_2)$ is not empty.

\noindent {\it Proof of the Claim}. First we note that the distance between $\phi (\epsilon)$ and $\phi (\lambda \epsilon)$ satisfies
$$
\vert \phi (\epsilon) - \phi (\lambda \epsilon) \vert \leq \vert \epsilon^{k_0+1} \Psi_1(\epsilon) - (\lambda\epsilon)^{k_0+1}
\Psi_1 (\lambda \epsilon) \vert \leq 2 C_1 \vert \epsilon^{k_0 +1} \vert \, .
$$
On the other hand, Koebe's quarter theorem \cite{Duren_book} implies that $\phi (B_1)$ contains a ball centered at
$\phi (\epsilon)$ of radius
$$
\frac{1}{4} \frac{\vert \epsilon \vert}{2} \vert \phi' (\epsilon) \vert = \frac{1}{8} \vert \epsilon \vert
O (\vert \epsilon^{k_0-1} \vert ) = O (\vert \epsilon^{k_0} \vert ) \, .
$$
Analogously we check that $\phi (B_2)$ also contains a ball centered at
$\phi (\lambda \epsilon)$ of radius $O (\vert \epsilon^{k_0} \vert )$. Since the distance from $\phi (\epsilon)$ to
$\phi(\lambda \epsilon)$ is bounded by a constant times $\vert \epsilon^{k_0 +1} \vert$ ($k_0$ being strictly negative), there follows that
$\phi (B_1) \cap \phi (B_2) \neq \emptyset$ provided that $\epsilon$ is sufficiently small. This establishes the claim and
completes the proof of Lemma~\ref{k_0isatworst-1}.
\end{proof}

Next, recalling that $\phi$ locally coincides with the developing map $\Dv_{\mu}$, there follows that the equation
$\xi = \Sch_{z} \phi$ holds on $\DD \setminus l^{3\pi/2}$. Furthermore $\phi$ is given by
$c \ln (z/z_0) /2\pi i + g(z) = c \ln (z/z_0) /2\pi i + \sum_{k=-1}^{\infty} a_k z^k$. Thus, we obtain:

\begin{lm}\label{order_phi_givenghasorderatleast-1}
	The order of $\xi$ at $0 \in \CC$ is greater than or equal to~$-1$.
\end{lm}

\begin{proof}
The proof is a direct computation similar to the proof of Lemma~\ref{prp:holomorphic_extension_trivial_mono}.
Owing to Lemma~\ref{k_0isatworst-1}, $g(z) = \sum_{k=-1}^{\infty} a_k z^k$. Since
$\phi =  c \ln (z/z_0) /2\pi i + \sum_{k=-1}^{\infty} a_k z^k$, it follows that
$$
\phi'(z) = \frac{c}{2\pi iz} + \sum_{k=-1}^{\infty} ka_kz^{k-1} = - \frac{a_{-1}}{z^2}+ \frac{c}{2\pi iz} \; + \; a_1
+ \sum_{k=2}^{\infty} ka_k z^{k-1}
$$
and
$$
\phi''(z) = \frac{2a_{-1}}{z^3} - \frac{c}{2\pi iz^2} + \sum_{k=2}^{\infty}k(k-1) a_k z^{k-2} \, .
$$
It suffices to treat the case $a_{-1}\neq 0$ since the computations become actually simpler for $a_{-1} = 0$.
For $a_{-1}\neq 0$, we have
$$
\phi'(z) = -\frac{a_{-1}}{z^2}\lb1 \;-\; \frac{cz}{2a_{-1}\pi i}+ z^2 h_1 (z)\rb
\;\;\;\; \text{ and } \;\;\;\;
\phi''(z) = \frac{2a_{-1}}{z^3} \lb 1 - \frac{cz}{4a_{-1}\pi i}  +z^3 h_2 (z)\rb
$$
for suitable holomorphic functions $h_1$ and $h_2$. Also, there exists $h_3$ holomorphic such that
$$
(1\;-\; \frac{cz}{2a_{-1}\pi i}+ z^2 h_1 (z))^{-1} = (1 \;+\; \frac{cz}{2a_{-1}\pi i} + z^2 h_3 (z)) \, .
$$
Hence,
\begin{equation}\label{goodequation-finding-xi}
\frac{\phi'' (z)}{\phi' (z)} = \frac{-2}{z} \lb 1 - \frac{cz}{4a_{-1}\pi i}   +z^3 h_2 (z) \rb \lb 1
\;+\; \frac{cz}{2a_{-1}\pi i}+ z^2 h_3 (z)\rb \, .
\end{equation}
Finally, from Equation~(\ref{goodequation-finding-xi}), we conclude that the quotient $\phi''(z)/\phi'(z)$ equals
$-2/z + h_4 (z)$, where $h_4$ is holomorphic. Hence
$$
\lb \frac{\phi''(z)}{\phi'(z)} \rb' = \frac{2}{z^2} + h_4' (z)
$$
where $h_4'$ is holomorphic. In turn,
\begin{eqnarray*}
\lb \frac{\phi''(z)}{\phi'(z)} \rb^2 & = & \frac{4}{z^2} \lb 1  \;-\; \frac{cz}{2a_{-1}\pi i} + z^2 h_5 (z) \rb \lb 1
\;+\; \frac{cz}{a_{-1}\pi i} + z^2 h_6 (z) \rb \\
& = & 	\frac{4}{z^2} \;+\; \frac{2c}{a_{-1}\pi i z} + h_7 (z) \, ,
\end{eqnarray*}
where $h_5, h_6, h_7$ are all holomorphic functions. Thus, in view of Formula~(\ref{eq:phi}), we conclude that
$$
\xi (z) = \Sch_{z} \phi = -\frac{c}{a_{-1}\pi i z}+ h_8 (z)
$$
for a suitable holomorphic function  $h_8$. This proves the lemma.
\end{proof}

We are now ready to derive Theorem~A and Corollary~B.

\begin{proof}[Proof of Theorem A and of Corollary B]
Consider a uniformizable singular projective structure $\Ps$ on a Riemann surface orbifold $\mc{L}$. Let $p \in \mc{L}$
be a singular point of $\Ps$ and denote by $\gamma \in \PSLR$ the local monodromy of $\Ps$ arising from a small loop
around $p$. According to Lemma~\ref{lm:UPS_disc*_not-hyp}, $\gamma$ cannot be hyperbolic. On the other hand, if $\gamma$ is elliptic,
it must be of finite order. In this case, and only in this case, $p$ is a singular point of orbifold type. The local structure
of $\Ps$ around $p$ is described by Proposition~\ref{prp:generalorbifoldform_ellipticmono}. Finally, if $\gamma$ is parabolic,
the corresponding statement in Theorem~A follows from Lemma~\ref{order_phi_givenghasorderatleast-1} along with its proof.

As to Corollary~B. The local finiteness of the norms $L^1$ and $L^{\infty}$ follows again from Proposition~\ref{prp:generalorbifoldform_ellipticmono}
in the case of elliptic monodromy. In turn, when the monodromy is parabolic, the statement is a consequence
of Lemma~\ref{order_phi_givenghasorderatleast-1}. This completes the proof of Corollary~B.
\end{proof}

\section{Proof of Proposition~\ref{Key_proposition_oninjectivedevelopingmaps}}\label{sec:Appendix}

The method to be employed in this section to prove Proposition~\ref{Key_proposition_oninjectivedevelopingmaps}
seems to be rather flexible and can be used to handle other similar situations. As an exemple, we state
Proposition~\ref{prp:inject_ess_sing} (see Remark~\ref{alternateproposition}) that also appears in problems about vector fields
having univalued solutions, see for example \cite{adolfoandI} and its references.
Also, let us point out that
Proposition~\ref{Key_proposition_oninjectivedevelopingmaps} is a special case of a nice conjecture put forward
in \cite{elsner}: it might be interesting to check how far our method can be pushed to provide insight in
Elsner's conjecture.

We fix a domain $\DD \setminus l^{\theta} \subset \CC$ where a branch of logarithm is defined. Also we have a holomorphic function
$g$ defined on $\DD^{\ast}$ and a multivalued function $\phi$ on $\DD^{\ast}$ which is defined by
\begin{equation}
\phi = g + \frac{c}{2\pi i} \ln (z/z_0)  \label{monodromy_for_phi-andln}
\end{equation}
for some $c \in \CC$ and $z_0 \in \DD^{\ast}$. Upon restriction to $\DD \setminus l^{\theta}$, all these
functions become well defined and $\phi$ is assumed to be injective for every fixed value of $\theta$. Clearly, as $\theta$ varies,
both $\phi$ and the branch of the logarithm can be continued in such way
that Equation~(\ref{monodromy_for_phi-andln}) will still hold for every value of $\theta$.
In other words, for every value of $\theta$ the map $\phi$ on $\DD \setminus l^{\theta}$ defined by a suitable choice of
a branch of logarithm satisfies~(\ref{difference_between_limits}) and is injective on $\DD \setminus l^{\theta}$.

Proposition~\ref{Key_proposition_oninjectivedevelopingmaps} will follow from the following proposition:

\begin{prp}\label{thereislimitforphi}
Let $\phi$ and $g$ be as above and assume that $\phi$ is injective on domains of the form $\DD \setminus l^{\theta}$
(cf. Proposition~\ref{Key_proposition_oninjectivedevelopingmaps}). Then
there exists the limit
$$
\lim_{z \rightarrow 0} \phi (z)
$$
with $z \in \DD \setminus l^{\theta}$.
\end{prp}

\begin{proof}[Proof of Proposition~\ref{Key_proposition_oninjectivedevelopingmaps}]
From Equation~(\ref{monodromy_for_phi-andln}), we have $2\pi i \phi /c = 2\pi i g/c + \ln (z/z_0)$. By taking the exponential on
both sides, we conclude that
$$
\exp \left( \frac{2\pi i \phi (z)}{c} \right) = \frac{z}{z_0} \exp \left( \frac{2 \pi i g(z)}{c} \right) \, .
$$
In particular, $\exp (2 \pi i \phi (z)/c)$ is holomorphic on a punctured neighborhood of $0 \in \CC$. However,
Proposition~\ref{thereislimitforphi} ensures that
$z e^{2 \pi i g(z)/c}$ has a limit as $z \rightarrow 0$ so that Casorati-Weierstrass theorem shows that $g$ cannot have an essential
singularity at $0 \in \CC$.
\end{proof}

The remainder of this section is devoted to the proof of Proposition~\ref{thereislimitforphi}.
Given real numbers $r, R$, such that $0 < r < R < 1$, the annulus of radii $r$ and $R$ around $0 \in \CC$ will be denoted by
$A_{(r,R)}$.

Next, we have that $\phi$ is injective on $\DD \setminus l^{\theta}$ and this still holds as $\theta$ varies and $\phi$ is
suitably continued, as previously indicated. Let now $\theta$ be fixed and choose $t >0$ arbitrarily small. Set
\[
A_{(r,R, \theta)}^t = \lcb \rho e^{i\alpha} : \, \, \rho, \alpha \in \RR \;, \;  r < \rho < R \; , \; \theta + t < \alpha < \theta + 2\pi - t \rcb \ .
\]
In other words, $A_{(r,R, \theta)}^t$ consists of those points in $A_{(r,R)}$ whose argument does not lie in $[\theta - t , \theta +t]$
(modulo $2\pi$). In particular, for $R$ small enough, we have $A_{(r,R, \theta)}^t \subset \DD \setminus l^{\theta}$.
Now let $\Lambda : \CC \longrightarrow \CC$
denote the map defined by $\Lambda (z) = rz/R$. The positive iterates of $\Lambda$ form the family $\{\Lambda^{\circ k}\}_{k \in \NN}$
and it is clear that $\Lambda(A_{(r,R, \theta)}^t) = A_{(r^2/R,r, \theta)}^t$. Although $A^{t}_{(r,R, \theta)}$ and
$\Lambda(A^t_{(r,R, \theta)})$ are disjoint sets, their closures have non-empty intersection. More precisely
$\overline{A_{(r,R,\theta)}^t} \cap \overline{\Lambda(A_{(r,R,\theta)}^t)}$ consists of those complex numbers having modulus~$r$
and argument in the interval $[\theta + t, \theta + 2\pi -t]$. Furthermore
\[
\overline{A_{(r,R,\theta)}^t} \cup \overline{\Lambda(A_{(r,R,\theta)}^t)} = \overline{A_{(r^2/R,R,\theta)}^t} \, .
\]
More generally, one has $\Lambda^{\circ k} (A_{(r,R,\theta)}^t) = A_{(r^{k+1}/R^k,r^k/R^{k-1},\theta)}^t$. Summarizing,
the family of pairwise disjoint open sets $\{A_{(r^{k+1}/R^k,r^k/R^{k-1},\theta)}^t\}_{k \in \NN_0}$ is such that
the union of their closures $\{\overline{A_{(r^{k+1}/R^k,r^k/R^{k-1},\theta)}^t}\}_{k \in \NN_0}$
yields a sector of angle $2\pi - 2t$ with vertex at $0 \in \CC$. In fact, the following holds:
\[
\bigcup_{k \geq 0} \overline{\Lambda^{\circ k}(A_{(r,R,\theta)}^t)} = \overline{B_0(R)} \setminus \{ \rho e^{i\alpha} : \, \, \rho, \alpha
\in \RR \;, \; 0 < \rho \leq R \; , \; \theta-t < \alpha < \theta +  t\}
\]
where $\overline{B_0(R)}$ stands for the closure of the disc $B_0 (R)$ of radius~$R$ around $0 \in \CC$. The interior of the
set $\overline{B_0(R)} \setminus \{ \rho e^{i\alpha} : \, \, \rho, \alpha
\in \RR \;, \; 0 < \rho \leq R \; , \; \theta-t < \alpha < \theta + t\}$ will be denoted by $B_0^t (R ,\theta)$.

From now on, let $R$ be such that $B_0 (2R)$ is contained in $\DD$, where $B_0 (2R)$ is the disc of radius $2R$ around $0 \in \CC$.
In particular, $\phi$ is injective on $B_0 (2R) \setminus l^{\theta}$.

Consider the sequence of maps $\{\xi^k\}$ defined on $A_{(r,R,\theta)}^t$ by
\[
\xi^k = \phi \circ \Lambda^{\circ k} \; .
\]
The maps $\xi^k$ are viewed as taking values in $\CC \PP^1 \simeq \CC \cup \{ \infty \}$. We have:

\begin{lm}\label{lm:normal_family}
The maps $\{\xi^k\}$ (taking values in $\CC \PP^1$) form a normal family for the compact open topology
\end{lm}

\begin{proof}
Recall that $\phi$ is injective on $B_0 (2R) \setminus l^{\theta}$. Similarly, the maps $\{\xi^k\}$ are also injective
on $A_{(r,R,\theta)}^t$. In particular, the sets $\xi^{k_1} (A^t_{(r,R,\theta)})$ and
$\xi^{k_2} (A^t_{(r,R,\theta)})$ are disjoint provided that $k_1 \neq k_2$. Next let $U$ be a small disc contained
in $A_{(R,2R)} \cap B_0 (2R) \setminus l^{\theta}$ and set $V = \phi (U)$.
Since $\phi$ is injective on $B_0 (2R) \setminus l^{\theta}$, it follows that $\xi^{k} (A^t_{(r,R)})$ is disjoint from
$V$ for every $k \in \NN_0$. In fact, the latter sets are all contained in
$\phi (B_0 (R) \setminus l^{\theta})$ which is itself disjoint from $V$. Up to choosing an appropriate coordinate on $\CC \PP^1$,
we can assume without loss of generality that $\phi (B_0 (R) \setminus l^{\theta})$
is contained in some disc $\Omega \subset \CC$. Thus the images of all the maps $\xi^k : A_{(r,R,\theta)}^t \rightarrow \CC \PP^1$
are actually contained in $\Omega$. The lemma now follows from Montel theorem.
\end{proof}

Next, we also have:

\begin{lm}\label{lm:seq_to_cte}
Every convergent subsequence of $\{\xi^k\}$ converges to a constant map.
\end{lm}

\begin{proof}
Since $\{\xi^k\}$ is a sequence of injective (holomorphic) maps, the limit of any uniformly convergent subsequence
is either a constant map or another (injective) holomorphic map
(see, for example, page 5 of \cite{Duren_book}). Let us then choose a convergent
subsequence $\{\xi^{k_n}\}$ of $\{\xi^k\}$ and assume for a contradiction that its limit $\xi^{\infty}$ is not constant. In this case,
the image of
$A^t_{(r,R,\textcolor{red}{\theta})}$ under $\xi^{\infty}$ must contain some open ball $B$ which, in turn, satisfies $B = \xi^{\infty} (U)$ for some open set
$U \subset A_{(r,R,\theta)}^t$.
Now, the fact that the sequence $\{\xi^{k_n}\}$ is uniformly convergent on $U$ and that the corresponding images converge to $B$
implies that $\xi^{k_{n_1}} (U) \cap \xi^{k_{n_2}} (U) \neq \emptyset$ for $n_1$ and $n_2$ sufficiently large. This is, however, impossible
since the images of $A^t_{(r,R,\theta)}$ by the iterates of $\xi$ are pairwise disjoint. The lemma is proved.
\end{proof}

Now, denote by $\mc{P}_0$ the set of points $p \in \CC \cup \{ \infty \}$ satisfying the following condition:
$p = \lim \phi (z_n)$ for some sequence
$\{z_n\}$ in $B_0 (R) \setminus l^{\theta}$ converging to $0\in\CC$. Furthermore, for every sequence
$\{z_n\} \subset B_0^t(R,\theta)$ converging to $0 \in
\CC$, we define a sequence $\{z'_n\} \subset A^t_{(r,R,\theta)}$ along with a sequence of positive integers $\{k(n)\}$
by means of the equation
\begin{align*}
z_n = \Lambda^{\circ k(n)} (z'_n) \ .
\end{align*}
Naturally, up to passing to a subsequence, we can assume without loss of generality that the sequence $\{z'_n\}$ converges to
some point in the closure of $A^t_{(r,R, \theta)}$. Let us denote by $\mc{P}_0'$ the subset of $\mc{P}_0$ consisting of the accumulation
points of $\phi$ arising from sequences $\{z_n\}$ satisfying the following condition: the sequence $\{z'_n\}$ is convergent and
its limit does not lie in the boundary $\partial A_{(r,R, \theta)}^t$ of $A_{(r,R, \theta)}^t$.
This condition ensures that $\{z'_n\}$ will be contained in some compact subset of $A_{(r,R, \theta)}^t$.

\begin{lm}\label{lm:Ac'_countable}
$\mc{P}_0'$ is a countable set.
\end{lm}

\begin{proof}
Let $p$ be a point in $\mc{P}_0'$ and $\{z_n\}$ a sequence as above such that $\lim \phi (z_n) = p$. Clearly, we have
$\phi (z_n) = \phi \lb \Lambda^{\circ k(n)} (z'_n) \rb = \xi^{k(n)} (z'_n)$. Owing to Montel theorem, up to passing to suitable subsequences,
we can assume that the sequence of maps $\{\xi^{k(n)}\}$ converges on some compact subset of $A_{(r,R,\theta)}^t$ containing all the points
in the sequence $\{z'_n\}$. Since Lemma~\ref{lm:seq_to_cte} states that the limit of $\{\xi^{k(n)}\}$ on a compact set is constant,
there follows that $p$ is,
in fact, the limit of some subsequence of $\{\xi^{k}\}$ on a compact subset of $A_{(r,R,\theta)}^t$.
Now, since there are only countably many subsequences of $\{\xi^{k}\}$, and
since we can fix a (countable) exhaustion of $A_{(r,R,\theta)}^t$ by compact sets, we conclude that only countably many of these
accumulation points can exist. This proves the lemma.
\end{proof}

Now Lemma \ref{lm:Ac'_countable} can be strengthened to encompass the entire set of accumulation points $\mc{P}_0$ of $\phi$.

\begin{lm}\label{lm:Ac_countable}
$\mc{P}_0$ is a countable set.
\end{lm}

\begin{proof}
Clearly, we only need to consider those sequences $\{z_n\} \in B_0 (R) \setminus l^{\theta}$ converging to $0\in\CC$ that
fall short of verifying the condition of Lemma~\ref{lm:Ac'_countable}.
In other words, the limit of the corresponding sequence $\{z_n'\} \in A^t_{(r,R, \theta)}$ lies in
the boundary of $A^t_{(r,R, \theta)}$. To handle this situation, it suffices to note
that the sequence in question is contained in the compact set given as the closure of $A^t_{(r,R, \theta)}$.
Clearly the closure of $A^t_{(r,R, \theta)}$
is contained in some open set $A^{t'}_{(r',R', \theta)}$ for some real constants $t', r', R'$ such that $0 < t'
< t$, $0 < r' < r$ and $R < R' < 2R$ so that $\phi$ is still injective on the relevant sets (recall that
$\phi$ is injective on $B_0 (2R) \setminus l^{\theta}$).
The result follows by applying the previous lemma with respect to the open set $A_{(r',R', \theta)}^{t'}$.
\end{proof}

Let us now consider the set $\mc{AC}$ of accumulation points $p = \lim \phi (z_n)$ where $\{ z_n \} \rightarrow 0 \in \CC$ and
where $\phi$ can be changed by changing the branch of logarithm $\ln$.
We claim that $\mc{AC}$ still is a countable set. More precisely, consider two
distinct angles $\theta_1$ and $\theta_2$ and $t >0$.  For $i=1,2$, define the sectors $V_1$, $V_2$ by
$$
V_i = \lcb \rho e^{i \alpha} \; : \; \; \; 0 < \rho < R \; \; \; {\rm and} \; \; \;
\theta_i + t <\alpha< \theta_i+2\pi - t \rcb \, .
$$
Now, up to reducing $t$, there follows that $V_1 \cup V_2$ contains a punctured neighborhood of $0 \in \CC$.
The previous lemma shows that the set of points $p = \lim \phi (z_n)$ with $z_n \rightarrow 0$ and $\{ z_n\} \subset V_i$
($i=1,2$) is countable. Clearly, if $p = \lim \phi (z_n)$ for some sequence $\{ z_n \} \rightarrow 0 \in \CC$, then we still
have $p = \lim \phi (z_{n(k)})$ for a subsequence $\{ z_{n(k)}\}$ fully contained in either $V_1$ or $V_2$. Thus,
$\mc{AC}$ is contained in the union of accumulation points obtained from sequences contained in either $V_1$ or $V_2$, at least
up to the choice of a determination for $\ln$ (which, in turn, gives rise to a determination of $\phi$ itself).
Since there are only countably many possible choices of branches of logarithm, two of them differing by a translation,
there follows that
$\mc{AC}$ is a countable set as desired.

We are now ready to prove Proposition \ref{prp:inject_ess_sing}.

\begin{proof}[Proof of Proposition \ref{thereislimitforphi}]
The proof amounts to showing that $\mc{AC}$, in fact, consists of a single point once a branch of logarithm $\ln$ is chosen.
In other words, let $\theta$ be fixed and let $\ln$ denote a branch of logarithm on $\DD \setminus l^{\theta}$. Now,
consider a (small) $t > 0$, and set
$$
V = \lcb \rho e^{i \alpha} \; : \; \; \; 0 < \rho < R \; \; \; {\rm and} \; \; \; \theta + t <\alpha< \theta+2\pi - t \rcb \, .
$$
To prove the proposition it suffices to ensure the existence of
a single point $p \in \CC \PP^1$ such that
$p = \lim \phi (z_n)$ where $z_n \rightarrow 0$ with $\{ z_n \} \subset V$.

Without loss of generality, we can assume $\theta =0$. Denote by $\mc{P}_0 \subseteq \mc{AC}$, the set of points
$p \in \CC \PP^1$ such that $p =\lim \phi (z_n)$ for some sequence $z_n \rightarrow 0$ with $\{ z_n \} \subset V$.
The set $\mc{P}_0$ is countable since so is $\mc{AC}$. Assume for a contradiction that $\mc{P}_0$ contains
two distinct points $p, q \in \CC \PP^1$ and choose sequences $\{x_n\}$ and $\{y_n\}$ satisfying the following conditions:
\begin{itemize}
	\item Both $\{x_n\}$ and $\{y_n\}$ converge to $0 \in \CC$ and we have $\{x_n\} \subset V$ and $\{y_n\} \subset V$.
	
	\item $p = \lim \phi (x_n)$ and $q = \lim \phi (y_n)$.
\end{itemize}

Next, let $l_n$ be a path joining $x_n$ to $y_n$ and entirely contained in $V$. Moreover, since $\{x_n\}$ and $\{y_n\}$ converge to $0 \in \CC$,
we can find a decreasing sequence $\{ \delta_n\}$, $\delta_n \in \RR_+$, such that the following holds:
\begin{itemize}
	\item[(1)] For every $n$, the path $l_n$ is entirely contained in the disc $B_0 (\delta_n)$ of radius $\delta_n$
	around $0 \in \CC$.
	
	\item[(2)] The decreasing sequence $\{ \delta_n \}$ converges to $0$.
\end{itemize}

Since $p$ and $q$ are distinct points, there exists $\varepsilon > 0$ such that $q$ does not belong to the closed ball $B_p (\varepsilon)$
of radius $\varepsilon$ around $p$. We also fix $\epsilon >0$ such that the closed disc $B_q (\epsilon)$ of radius $\epsilon$ around $q$ lies
entirely in the complement of $B_p (\varepsilon)$ inside $\CC \PP^1$.

Now, since $p = \lim \phi (x_n)$ and $q = \lim \phi (y_n)$, there exists $n_0 \in \NN$ such that, for all $n \geq n_0$,
the following holds:
$$
\phi (x_n) \in B_p (\varepsilon) \; \; \; \; {\rm and} \; \;\; \; \phi (y_n) \in B_q (\epsilon) \, .
$$
Thus, for every $n$ large enough, there is a point $w_n \in \phi (l_n)$ such that $w_n \in \partial
B_q(\epsilon)$. The sequence $\{w_n\}$ is contained in a compact set so that it possesses at least one accumulation point
in $\partial B_q(\epsilon)$. Let $w_{\epsilon} \in \partial B_q(\epsilon)$ be one accumulation point of the
sequence $\{w_n\}$. Clearly, for every $n$, $w_n = \phi (z_n)$ for some $z_n \in l_n$. By construction, the sequence $\{ z_n\}$
is contained in $V$ and converges to $0 \in \CC$ so that $w_{\epsilon}$ lies in $\mc{P}_0$.

To complete the proof, just let $\epsilon$ vary on an arbitrarily small interval. The preceding construction then yields
a continuum of points in $\mc{P}_0$ hence contradicting the fact that $\mc{P}_0$ is countable. This ends
the proof of Proposition~\ref{prp:inject_ess_sing}.
\end{proof}

\begin{rmk}\label{alternateproposition}
{\rm Let us close this section by pointing out another result (Proposition~\ref{prp:inject_ess_sing})
that can also be proved by following essentially the same argument provided above for
Propositions~\ref{Key_proposition_oninjectivedevelopingmaps} and~\ref{thereislimitforphi}. With this
notation, Proposition~\ref{prp:inject_ess_sing} reads:

Consider a holomorphic function $\mc{H}$ defined on a punctured neighborhood of $0 \in \CC$. Given $\alpha \in \CC$,
let $f$ be the multivalued function defined by $f(z) = \mc{H}(z) \, z^{\alpha}$. In accurate terms, $f$ is a well-defined
holomorphic function on domains of the form $\DD \setminus l^{\theta}$ up to re-scaling coordinates and this applies
to every choice of $\theta \in [0, 2\pi[$.

\begin{prp}\label{prp:inject_ess_sing}
	Assume that the function $f(z) = \mc{H}(z) \, z^{\alpha}$ is injective on the sector $\DD \setminus l^{\theta}$,
	for every $\theta \in [0,2\pi)$.
	Then $\mc{H}(z)$ does not have an essential singularity at $0 \in \CC$.
\end{prp}

}
\end{rmk}

\section{Covering projective structures and proof of Theorem C}\label{sec:Cov_PS}

In this final section, we shall derive Theorem~C from well-known results in \cite{shiga} and \cite{kra3}. Throughout the
section $\mc{L}$ stands for a Riemann surface (orbifold) of finite type: this means that $\mc{L}$ is isomorphic to
a compact Riemann surface (orbifold) from which finitely many points were removed.

Recall that a (singular) bounded projective structure $\Ps$ on $\mc{L}$ is said to be a {\it covering projective structures}\,
if it satisfies the following conditions:
\begin{itemize}
	\item The developing map $\Dv$ (associated with the universal covering) is a covering of
	its image.
	
	\item The monodromy group of $\Ps$ acts discontinuously on the image of $\Dv$.
\end{itemize}

\begin{lm}\label{lm:UPS_to_CPS}
A uniformizable projective structure $\Ps$ on $\mc{L}$ is necessarily a bounded covering projective structure.
\end{lm}

\begin{proof}
The monodromy developing map $\Dv_\mu$ realizes a diffeomorphism between $\mc{L}_\mu$ and its image in $\CC \PP^1$.
If $\ker(\mu)$ is trivial then $\mc{L}_\mu$ is the universal covering and $\Dv=\Dv_\mu$ is, in particular, a covering map.
Let us now assume that $\ker(\mu)$ is non-trivial, and consider the covering map $h:\DD\longrightarrow\mc{L}_\mu$.
By construction, the (universal covering) developing map $\Dv$ and the monodromy-developing map $\Dv_\mu$
have the same image $U$ in $\CC\PP^1$ since
the monodromy of the lift of $\Ps$ to $\mc{L}_\mu$ is trivial. In other words, the developing map
$\Dv$ can be factored as a composition $\Dv=\Dv_\mu\circ h$. In particular, $\Dv$ is a composition of a diffeomorphism
($\Dv_\mu : \mc{L}_\mu \rightarrow U$) and a covering map ($h:\DD\longrightarrow\mc{L}_\mu$). Thus
$\Dv$ is a covering map itself. Finally, this projective structure $\Ps$ on $\mc{L}$ is necessarily bounded thanks
to Corollary~B.
\end{proof}

Strictly speaking the converse to Lemma~\ref{lm:UPS_to_CPS} does not hold in general. The simplest example of a covering projective structure that is
not uniformizable arises in $\CC$ (or equivalently in $\CC \PP^1$) by means of the singular projective structure $\Ps_k$ underlining the
translation structure induced by the vector field $z^k \partial /\partial z$, $k \geq 3$. Naturally, $\CC$ can be identified to its own
ramified $(k-1)$-sheet covering $\CC_{k-1}$ by means of the map $z \stackrel{\pi_k}{\mapsto} z^{k-1} \in \CC$. With this notation, the projective
structure $\Ps_k$ is nothing but the lift to $\CC_{k-1} \simeq \CC$ of the (translation) structure $\Ps_2$ induced on $\CC$ by the vector field
$z^2 \partial /\partial z$. Whereas $\Ps_2$ is uniformizable and has a singular point at $0 \in \CC$, it has no periods on $\CC^{\ast}$.
Thus, the monodromy covering associated with $\Ps_2$ is trivial and so is the monodromy covering associated with $\Ps_k$, $k\geq 3$.
However, for $k \geq 3$, the loop $c : [0,1] \rightarrow \CC^{\ast}$ given by $c(t) = e^{2\pi i t}$ lifts to $\CC_{k-1} \simeq \CC$ as an open
path whose endpoints have, by construction, the same image under the monodromy-developing map. In other words, the monodromy developing
map arising from $\Ps_k$ is not injective so that $\Ps_k$ is not uniformizable.

A partial converse to Lemma~\ref{lm:UPS_to_CPS}, however, is still possible on Riemann surfaces of finite type and this
is the content of Lemma~\ref{lm:CPS_to_UPS} below.

\begin{lm}\label{lm:CPS_to_UPS} Let $\mc{L}$ denote a Riemann surface (orbifold) of finite type and let $\Ps$ be a (singular) bounded covering
	projective structure on $\mc{L}$. Then, there exists a finite quotient $\mc{L}^{\ast}=\mc{L}/\sim$ of $\mc{L}$ where $\Ps$ induces a uniformizable
	projective structure. In other words, the pair $(\mc{L}, \Ps)$ is obtained as a finite ramified covering of a pair
	$(\mc{L}^{\ast}, \Ps_{\sim})$ where $\mc{L}^{\ast}$ is a Riemann surface (orbifold) of finite type and where $\Ps_{\sim}$ is a uniformizable
	projective structure on $\mc{L}^{\ast}$.
\end{lm}

\begin{proof} On the universal covering of $\mc{L}$, the developing map arising from $\Ps$ is a covering map of its image. The same
	then applies to the monodromy-developing map $\Dv_\mu$ from the monodromy covering $\mc{L}_{\mu}$ of $\mc{L}$ to
	its image in $\CC \PP^1$. Denote by ${\rm Im}\, (\Dv_\mu) \subset \CC \PP^1$ the image of $\Dv_\mu$. By definition, it is
	an open set invariant by the monodromy group $\Gamma$ of $\Ps$.

	Next, we claim that	the fibers of $\Dv_\mu$ are finite. To check this claim, consider a point $q_1 \in \mc{L}_{\mu}$
	lying in a fundamental domain $\mc{L}_0 \subset \mc{L}_{\mu}$ with respect to the covering $\mc{L}_{\mu} \rightarrow \mc{L}$.
	Let $U_0 \subset {\rm Im}\, (\Dv_\mu)$ be given by $U_0 = \Dv_\mu (\mc{L}_0)$. The equivariance relation verified
	by developing maps (Equation~(\ref{eq:dev-mon-rel})) implies that $U_0$ is a fundamental domain for the
	action of $\Gamma$ on ${\rm Im}\, (\Dv_\mu)$. In particular, the fact that the group of deck-transformations
	of the covering $\mc{L}_{\mu} \rightarrow \mc{L}$ is naturally identified with the group $\Gamma$ combined with
	Equation~(\ref{eq:dev-mon-rel}) to ensure that another point $q_2 \in \mc{L}_{\mu}$ satisfying
	$\Dv_\mu (q_1) = \Dv_\mu (q_2)$ must belong to the same fundamental domain $\mc{L}_0$ as $q_1$.

	Letting $z_0 = \Dv_\mu (q_1)$, the fiber $\Dv_\mu^{-1} (z_0)$ is therefore contained in $\mc{L}_0$ which is identified
	with the original Riemann surface (orbifold) $\mc{L}$. The fiber $\Dv_\mu^{-1} (z_0)$ is also a discrete set since
	$\Dv_\mu$ is a covering map. To conclude that $\Dv_\mu^{-1} (z_0)$ is a finite set, we now proceed as follows.

	Recall that $\mc{L}_0 \simeq \mc{L}$ is isomorphic to a compact Riemann surface (orbifold) $S$ minus finitely many
	points $p_1, \ldots , p_k$. The bounded nature of $\Ps$, however, ensures that a small neighborhood of any of the punctures
	$p_i$ can contain only finitely many points of the fiber $\Dv_\mu^{-1} (z_0)$: this follows from that $\Ps$ is
	of bounded type and therefore singular points of $\Ps$ are at worst poles of order~$1$ for the corresponding
	Schwarzian differential.
	Therefore, up to finitely many points, the fiber $\Dv_\mu^{-1} (z_0)$ yields a discrete set of a compact part
	of $S$. The finiteness of $\Dv_\mu^{-1} (z_0)$ follows at once.

	Next, consider the following equivalence relation in $\mc{L}_{\mu}$: two points in $\mc{L}_{\mu}$ are identified
	if they have the same image under $\Dv_\mu$. The preceding shows that equivalence classes are finite (and of same
	cardinality). Moreover the group of deck-transformations of the covering $\mc{L}_{\mu} \rightarrow \mc{L}$ sends
	equivalence classes to equivalence classes owing to Equation~(\ref{eq:dev-mon-rel}). In particular, we obtain an equivalence
	relation in $\mc{L}$ itself: two points $p, q$ are equivalent if they have lifts $\tilde{p}, \tilde{q}$ in
	$\mc{L}_{\mu}$ (necessarily belonging to a same fundamental domain) such that $\Dv_\mu (\tilde{p}) =
	\Dv_\mu (\tilde{q})$. Clearly, the quotient $\mc{L}^{\ast}$ of $\mc{L}$ by this equivalence relation still is
	a Riemann surface (orbifold) of finite type.
	
	It only remains to show that $\mc{L}^{\ast}$ is endowed with a uniformizable projective structure $\Ps_{\sim}$
	induced by $\Ps$. For this, recall that a projective structure can also be defined by the pair consisting
	of a developing map and a monodromy group satisfying Equation~(\ref{eq:dev-mon-rel}). In the present case, we simply
	identify the points of $\mc{L}_{\mu}$ having the same image under $\Dv_\mu$ obtaining a new manifold
	$\mc{L}_{\mu}^{\ast}$. By construction $\Dv_\mu$ induces a developing map $\Dv_{\mu, \sim}$ which is a diffeomorphism
	from $\mc{L}_{\mu}^{\ast}$ to ${\rm Im}\, (\Dv_\mu)$. Similarly, what precedes ensures that the group of deck transformations
	of $\mc{L}_{\mu}$ still acts on $\mc{L}_{\mu}^{\ast}$ and that the quotient of this action is nothing but $\mc{L}^{\ast}$.
	Finally, $\Dv_{\mu, \sim}$ is still equivariant with respect to the action of the mentioned deck-transformation group on
	$\mc{L}_{\mu}^{\ast}$ and the action of $\Gamma$ on ${\rm Im}\, (\Dv_\mu)$. It therefore defines a projective structure
	on $\mc{L}^{\ast}$ which is uniformizable since $\Dv_{\mu, \sim} : \mc{L}_{\mu}^{\ast} \rightarrow {\rm Im}\, (\Dv_\mu)$ is
	a diffeomorphism. The lemma is proved.
\end{proof}

We are now ready to derive Theorem~C from the analogous results in \cite{shiga} and \cite{kra3} about bounded covering projective
structures.

\begin{proof}[Proof of Theorem C]
	To prove the first assertion of Theorem~C, let $\mathcal{S} (\mc{L})$ denote the space of bounded covering projective structures
	on $\mc{L}$. Similarly, let $\mathcal{K} (\mc{L})$ be the set formed by bounded projective structures
	whose monodromy group is Kleinian with non-empty discontinuity set. Finally $\mathcal{U} (\mc{L})$ denotes the
	set of uniformizable projective structure on $\mc{L}$. Bers simulteneous uniformization theorem shows that
	quasi-conformal	deformations of the canonical projective structure on $\mc{L}$ is contained in the interior of $\mathcal{U} (\mc{L})$.
	To show the converse inclusion, we need the result of \cite{shiga} claiming that the interior of
	$\mathcal{K} (\mc{L}) \cap \mathcal{S} (\mc{L})$ actually coincides with the mentioned space of quasi-conformal	deformations. However,
	the interior of $\mathcal{U} (\mc{L})$ is contained in the interior of $\mathcal{K} (\mc{L}) \cap \mathcal{S} (\mc{L})$. In fact,
	Lemma~\ref{lm:UPS_to_CPS} yields $\mathcal{U} (\mc{L}) \subseteq \mathcal{S} (\mc{L})$ while we also have
	$\mathcal{U} (\mc{L}) \subseteq \mathcal{K} (\mc{L})$. For the latter assertion, recall from Lemma~\ref{lm:uniformizable_vs_quotient} that
	a Riemann surface with a uniformizable
	projective structure is the quotient of an open set of $\CC \PP^1$ - namely the image of its monodromy developing map - by
	its monodromy group. In particular, the monodromy group must be discrete and with non-empty discontinuity set.
	
	In turn, the second assertion of Theorem~C is an immediate consequence of the analogous statement in \cite{kra3} for bounded
	projective structures complemented by Lemma~\ref{lm:CPS_to_UPS}.
\end{proof}

\bigbreak

\begin{flushleft}
{\sc Ahmed Elshafei} \\
Centro de Matem\'atica da Universidade do Porto\\
Centro de Matem\'atica da Universidade do Minho\\
Portugal\\
a.el-shafei@hotmail.fr

\end{flushleft}

\bigskip

\begin{flushleft}
{\sc Julio Rebelo} \\
Institut de Math\'ematiques de Toulouse\\
118 Route de Narbonne\\
F-31062 Toulouse, FRANCE.\\
rebelo@math.univ-toulouse.fr

\end{flushleft}

\bigskip

\begin{flushleft}
{\sc Helena Reis} \\
Centro de Matem\'atica da Universidade do Porto, \\
Faculdade de Economia da Universidade do Porto, \\
Portugal\\
hreis@fep.up.pt \\
\end{flushleft}

\end{document}